%% file: chow.tex
\documentclass[10pt, a4paper]{amsart}

\pdfoutput=1

\usepackage[T1]{fontenc}

\usepackage{mathrsfs, amsmath, amsthm, amssymb, stmaryrd, enumerate}%
\usepackage{tikz}%
\usepackage[linktocpage]{hyperref}%
\usepackage[all]{xy}%
\xyoption{2cell}%
\UseAllTwocells%

\usepackage[centerboxes]{ytableau}

\newdir{t>}{{\UseTips\dir{>}}}

\DeclareMathOperator{\id}{id}
\DeclareMathOperator{\el}{el}
\DeclareMathOperator{\op}{op}
\DeclareMathOperator{\Lan}{Lan}

\DeclareMathOperator{\Rep}{Rep}
\DeclareMathOperator{\Vect}{\mathbf{Vect}}

\DeclareMathOperator{\fd}{fd}
\DeclareMathOperator{\Hom}{Hom}

\DeclareMathOperator{\Sh}{Sh}

\DeclareMathOperator{\Ind}{Ind}

\DeclareMathOperator{\lax}{lax}

\DeclareMathOperator{\colim}{colim}

\DeclareMathOperator{\Chb}{\mathrm{Ch}^{\mathrm{b}}}
\DeclareMathOperator{\Kb}{\mathrm{K}^{\mathrm{b}}}

\DeclareMathOperator{\Lex}{\mathbf{Lex}}

\DeclareMathOperator{\Ab}{\mathbf{Ab}}

\DeclareMathOperator{\Fun}{\mathrm{Fun}}

\newcommand{\ca}[1]{\mathscr{#1}}

\newcommand{\sh}{\mathrm{Sh}}

\DeclareMathOperator{\U}{\mathbb{U}}

\newcommand{\ten}[1]{\mathop{{\otimes}_{#1}}}

\newcommand{\defl}{\mathrel{\mathop:}=}

\theoremstyle{plain}
\newtheorem{thm}{Theorem}[subsection]
\newtheorem*{thm*}{Theorem}
\newtheorem{prop}[thm]{Proposition}

\newtheorem{cor}[thm]{Corollary}

\theoremstyle{definition}

\newtheorem{rmk}[thm]{Remark}
\newtheorem{dfn}[thm]{Definition}

\newtheoremstyle{citing}{}{}{\itshape}{}{\bfseries}{.}{ }{\thmnote{#3}}
\theoremstyle{citing}

\newtheoremstyle{citingdfn}{}{}{}{}{\bfseries}{.}{ }{\thmnote{#3}}
\theoremstyle{citingdfn}

\numberwithin{equation}{section}


\author{Daniel Sch{\"a}ppi}
\thanks{This research was supported by the DFG grant: SFB 1085 ``Higher invariants.''}
\address{Fakult{\"a}t f{\"u}r Mathematik,
Universit{\"a}t Regensburg,
93040 Regensburg,
Germany}
\email{daniel.schaeppi@ur.de}

\title{Graded-Tannakian categories of motives}

\begin{document}

\begin{abstract}
 
 Given a rigid tensor-triangulated category and a vector space valued homological functor for which the K{\"u}nneth isomorphism holds, we construct a universal graded-Tannakian category through which the given homological functor factors. We use this to (unconditionally) construct graded-Tannakian categories of pure motives associated to a fixed Weil cohomology theory, with a fiber functor realizing the given cohomology theory. For $\ell$-adic cohomology and a ground field which is algebraic over a finite field, this category is Tannakian. In this case, we obtain in particular motivic Galois groups which act naturally on $\ell$-adic cohomology without assuming any of the standard conjectures. We show that these graded-Tannakian categories are equivalent to Grothendieck's category of pure motives if the standard conjecture D holds.
 
\end{abstract}

\maketitle

\tableofcontents

\input{introduction}
\input{main}

\bibliographystyle{amsalpha}
\bibliography{chow}

\end{document}

%% file: introduction.tex
\section{Introduction}

If the standard conjectures hold, then Grothendieck's category of motives is a Tannakian category with a universal property: every Weil cohomology theory arises via a fiber functor from the universal Weil cohomology theory with values in motives. Beilinson and Grothendieck also conjectured that there should exist a Tannakian category of mixed motives (which would include motives of singular varieties). There are by now various triangulated categories of mixed motives over a field (due to Hanamura, Levine, Voevodsky, later extended to more general base schemes by various authors) which are conjecturally equivalent to the derived category of the Tannakian category of mixed motives. On the other hand, there are also several abelian categories of pure motives due to Deligne and Milne (in characteristic zero, using absolute Hodge cycles) \cite{DELIGNE_MILNE} and Andr{\'e}  (using motivated cycles) \cite{ANDRE}. Nori constructed an abelian category of mixed motives in characteristic zero.

 In this article, we give a systematic construction to move from triangulated categories (of motives) to abelian categories. More precisely, we can ask the following question. Given a triangulated category $\ca{T}$ and a homological functor $H$ from $\ca{T}$ to graded vector spaces, does the homological functor factor through a universal abelian category? In \S \ref{section:triangulated}, we give an affirmative answer to this question if $\ca{T}$ is a rigid tensor-triangulated category and the homological functor $H$ is symmetric strong monoidal (that is, the K{\"u}nneth isomorphism holds): there exists a rigid symmetric monoidal abelian category $\ca{M}_H$, a symmetric strong monoidal functor $[-] \colon \ca{T} \rightarrow \ca{M}_H$, and a faithful exact functor $\bar{H}$ to graded vector spaces such that $H\cong \bar{H}[-]$. We use the name \emph{graded-Tannakian} for categories such as $\ca{M}_{H}$ and \emph{graded fiber functor} for $\bar{H}$ (in accordance with the usage of the term \emph{graded-commutative} in algebraic topology). In characteristic zero, Deligne has given an intrinsic characterization of such categories in terms of the behaviour of Schur functors \cite{DELIGNE_TENSORIELLE}.

 The categroy $\ca{M}_H$ has the following universal property: any other symmetric strong monoidal homological functor $H^{\prime}$ on $\ca{T}$ which has the same strength as $H$ (meaning $Hf=0$ if and only if $H^{\prime}f=0$) arises essentially uniquely from a graded fiber functor on $\ca{M}_H$. The condition that $H$ and $H^{\prime}$ have the same strength is clearly necessary since fiber functors are faithful. The construction gives in particular candidates for graded-Tannakian categories of mixed motives, starting with various triangulated categories of mixed motives and their realization functors. For example, each mixed Weil cohomology theory (on smooth schemes over a perfect field) in the sense of Cisinski and D\'{e}glise \cite{CISINSKI_DEGLISE} has a naturally associated graded-Tannakian category (see Corollary~\ref{cor:mixed_weil}).
 
 If there exists a motivic $t$-structure with very good properties with respect to a given realization $H$, we get a comparison functor $\ca{M}_H \rightarrow \ca{M}$ directly from the universal property. In Corollary~\ref{cor:derived} we give further evidence that the categories $\ca{M}_H$ should provide good candidates for a Tannakian category of mixed motives. We leave a detailed study of these examples to future work and focus attention on categories of pure motives instead.

In \S \ref{section:applications}, we apply the general construction in the case where the triangulated category $\ca{T}$ is the homotopy category of bounded complexes of Chow motives and investigate some of the properties of the resulting graded-Tannakian category.

 Let $k$ be a field, $K$ a field of characteristic zero and $H$ a Weil cohomology theory with values in $K$-vector spaces. We write $\mathrm{Mot}_H(k)$ for the category of cohomological Chow motives modulo homological equivalence (see \cite{MANIN}, \cite[\S 4]{DEMAZURE}). From this data, we construct a graded-Tannakian category $\ca{M}_H(k)$ with a graded fiber functor $\bar{H}$ and a  symmetric strong monoidal functor $[-]\colon \mathrm{Mot}_H(k) \rightarrow \ca{M}_H(k)$ such that $H\cong \bar{H}[-]$. Recall that the standard conjecture $D$ for $H$ asserts that $H$-equivalence coincides with numerical equivalence. In \S \ref{section:applications}, we prove the following theorem (see Theorem~\ref{thm:properties} for precise statements).

 \begin{thm*}
 The category $\ca{M}_{H}(k)$ and the functors $[-]$, $\bar{H}$ have the following properties.
 \begin{enumerate}
 \item[(i)]  If the standard conjecture $D$ holds for $H$, then $[-] \colon \mathrm{Mot_H}(k) \rightarrow \ca{M}_H(k)$ is an equivalence.
  \item[(ii)] If $H^{\prime}$ is another Weil cohomology theory taking values in $K^{\prime}$-vector spaces and there is a common field extension of $K$ and $K^{\prime}$ over which $H$ and $H^{\prime}$ are isomorphic, then there exists a graded fiber functor $\bar{H^{\prime}}$ from $\ca{M}_H(k)$ to $K^{\prime}$-vector spaces such that $H^{\prime} \cong \bar{H^{\prime}}[-]$.
 \item[(iii)] If $k$ has characteristic zero, $H$ is classical (de Rham, $\ell$-adic {\'e}tale, or Betti cohomology for $k \subseteq \mathbb{C}$), and $\ca{M}_k$ is Andr{\'e}'s Tannakian category defined via motivated cycles \cite[\S 4]{ANDRE}, then there exists a (non symmetric) strong monoidal functor
  \[
  \ca{M}_H(k) \rightarrow \ca{M}_k
  \]
  which is faithful and exact and commutes with the respective realizations up to (non-symmetric) monoidal isomorphism.
\end{enumerate}   
 \end{thm*}
 
 The reason for the incompatibility with symmetries is that it is unclear if every object in $\ca{M}_{H}(k)$ admits a direct sum decomposition with summands $M_i$ pure of weight $i$ (that is, such that $\bar{H}^iM_i$ is concentrated in degree $i$), so it is not clear if one can twist the signs of the symmetry isomorphisms. One case where this \emph{can} be carried out is for $k$ algebraic over a finite field and $H$ $\ell$-adic {\'e}tale cohomology (using a consequence of Deligne's solution of the Weil conjectures due to Katz and Messing). Thus we obtain (unconditionally) Tannakian categories $\ca{M}_{\text{{\'e}t},\ell}(k)\defl\ca{M}_{H_{\text{{\'e}t},\ell}}^{\mathrm{tw}}(k)$ of pure motives with a fiber functor $\bar{H}$ to $\mathbb{Q}_{\ell}$-vector spaces realizing $\ell$-adic cohomology. This category differs from the existing Tannakian categories of pure motives in positive characteristic (defined for example in  \cite{MILNE} and\cite[\S 9]{ANDRE}) since the latter are not known to have such a realization functor without assuming some conjectures. For example, the pro-reductive motivic Galois groups of \cite{ANDRE_KAHN} naturally act on a Weil cohomology theory which is not known to be isomorphic to $\ell$-adic cohomology; on the other hand, the motivic Galois groups obtained from $\ca{M}_{\text{{\'e}t},\ell}(k)$ are not known to be pro-reductive. The relationship with motivated cycles is reversed from the case of characteristic zero: every motivated cycle relative to $\ell$-adic {\'e}tale cohomology comes from a morphism in $\ca{M}_{\text{{\'e}t},\ell}(k)$ if $k$ is algebraic over a finite field (this follows from the hard Lefschetz theorem for $\ell$-adic cohomology proved by Deligne). Note that the categories $\ca{M}_{\text{{\'e}t},\ell}(k)$ are only conjecturally independent of the prime $\ell$.

 Since the functor $[-]$ is not necessarily full if the standard conjecture $D$ does not hold for $H$, we do get variations of the usual standard conjectures by asking that certain morphisms exist in $\ca{M}_H(k)$ rather than in $\mathrm{Mot}_H(k)$. We investigate various relationships between these ``weak'' standard conjectures and their structural implications for $\ca{M}_H(k)$ in \S \ref{section:standard}.

 In characteristic zero, there are also constructions of motivic Galois groups due to Ayoub and Nori, which coincide by work of Choudhury and Gallauer \cite{AYOUB,CHOUDHURY_GALLAUER}. It would be interesting to see if the methods of \cite{CHOUDHURY_GALLAUER} can be used to relate Nori's Tannakian category to the universal graded-Tannakian obtained from  the Betti realization and to understand the relationship between the various motivic Galois groups in characteristic zero.
 
 The construction of $\ca{M}_H$ is closely related to the construction of examples of $\mathrm{tt}$-fields by Balmer, Krause, and Stevenson \cite{BKS}, though the resulting objects are not quite the same. We do not expect that $\ca{M}_H$ always satisfies the maximality property required for the construction of $\mathrm{tt}$-fields. On the other hand, existence of fiber functors on $\mathrm{tt}$-fields (more precisely, on the abelian category appearing in the construction) is not discussed in \cite{BKS}. It would be interesting to understand the precise relationship between these constructions. It would also be interesting to relate the results about universal graded-Tannakian categories to derived versions of Tannaka duality due to Iwanari \cite{IWANARI} and Pridham \cite{PRIDHAM}.
 
  \section*{Acknowledgements}
 This project relies on a construction due to Piotr Pstr\k{a}gowski. I am very grateful to Piotr for explaining his construction and some of its interesting examples to me. I thank Niko Naumann for helpful discussions and encouragement and I thank Denis-Charles Cisinski and Marc Hoyois for suggesting several improvements. Support of the SFB 1085 ``Higher invariants'' in Regensburg is gratefully acknowledged.

%% file: main.tex
\section{Graded-Tannakian categories associated to K{\"u}nneth functors}\label{section:triangulated}

\subsection{The underlying category}
 Let $K$ be a field. Throughout this section we fix an abelian group $A$ and a group homomorphism $\varepsilon \colon (A,+) \rightarrow (\mathbb{Z}^{\times},\cdot)$, and we write $(A,\varepsilon)\mbox{-}\Vect_K$ for the category of $A$-graded $K$-vector spaces with the Koszul symmetric monoidal structure: the symmetry isomorphism is constructed via $\varepsilon$ following the Koszul sign rules. More generally, given an additive symmetric monoidal category $\ca{A}$, we write $(A,\varepsilon)\mbox{-}\mathrm{gr}(\ca{A})$ for the corresponding symmetric monoidal category of $A$-graded objects in $\ca{A}$.
 
 Let $\ca{T}$ be a small triangulated category and suppose that $\ca{T}$ has a symmetric monoidal structure $(\ca{T},\otimes, \U)$. We demand that the following two compatibilities hold between the triangulated structure and the monoidal structure.
 
 \begin{dfn}
 Let $S^1=\U[1]$ be the shift of the unit object. Then $\ca{T}$ is called \emph{$\otimes$-triangulated} if there is a natural isomorphism $(-)[1] \cong -\otimes S^1$ and for each $X \in \ca{T}$, the functor $X \otimes -$ preserves distinguished triangles.
 
 Let $\ca{A}$ be a symmetric monoidal abelian category. A \emph{K{\"u}nneth functor} is a homological functor $H \colon \ca{T} \rightarrow \ca{A}$ which is symmetric strong monoidal.
 \end{dfn}

 Note that this in particular implies that $S^1$ is an invertible object in $\ca{T}$. We also assume that $\ca{T}$ is \emph{rigid}, that is, every object $X \in \ca{T}$ has a dual $X^{\vee}$ (meaning that $X \otimes - $ is left adjoint to $X^{\vee} \otimes -$). 
 Now suppose that $H \colon \ca{T} \rightarrow (A,\varepsilon)\mbox{-}\Vect_K$ is a K{\"u}nneth functor. Since $H$ preserves duals, the assumption that $\ca{T}$ is rigid implies that $HX$ is zero for all but finitely many degrees and always finite dimensional over $K$, in other words, $HX$ has finite total dimension.
 
 Given an additive category $\ca{A}$, we write $[\ca{A}^{\op},\Ab]$ for the category of additive presheaves on $\ca{A}$. Given a functor $F \colon \ca{A} \rightarrow \ca{C}$ to a cocomplete additive category $\ca{C}$, we get an induced left adjoint
 \[
  - \ten{\ca{A}} F \colon [\ca{A}^{\op},\Ab] \rightarrow \ca{C}
 \]
 which is known as either the \emph{functor tensor product} or \emph{left Kan extension} along the Yoneda embedding (and then denoted by $\Lan_Y F$). We write $\Hom_{\ca{A}}(F,-)$ for its right adjoint, which sends $C \in \ca{C}$ to the presheaf $\ca{C}(F-,C)$. Another common notation for this is $\widetilde{F}$.
 
 \begin{prop}\label{prop:exactness}
 Let $H \colon \ca{T} \rightarrow (A,\varepsilon)\mbox{-} \Vect_K$ be a homological functor. Then $H$ is \emph{flat}: the functor
 \[
 -\ten{\ca{T}} H \colon [\ca{T}^{\op},\Ab] \rightarrow (A,\varepsilon)\mbox{-} \Vect_K
 \]
 is left exact.
 \end{prop}
 
 \begin{proof}
 The functors which send a graded vector space $(V_i)_{i \in A}$ to the underlying abelian group of $V_i$ preseve all colimits, in particular functor tensor products, and they preserve exact sequences. Since they jointly detect isomorphisms, we are reduced to the case of a homological functor $H \colon \ca{T} \rightarrow \Ab$. Here we can use a well-known generalizations of Lazard's theorem about flat modules: if $H$ is a filtered colimit of representable functors $\ca{T}(-,X)$ (the ``finitely generated free modules'' in this context), then $H$ is flat. Since $\ca{T}$ has finite direct sums, it follows that $H$ is the colimit of the slice category of representable functors over $H$ in $[\ca{T},\Ab]$. This is equivalent to the opposite of the category $\el(H)$ of elements of $H$, with objects the pairs $(X,x \in HX)$ and morphisms $(X,x)\rightarrow (Y,y)$ the morphisms $f \colon X \rightarrow Y$ in $\ca{T}$ such that $Hf(x)=y$. Non-emptiness and the existence of the desired spans in this category is immediate from the fact that $\ca{T}$ has finite direct sums. It only remains to check that for any pair $f,g \colon X\rightarrow Y$ with $Hf(x)=Hf(y)$, there exists a morphism $h \colon V \rightarrow X$ and $v \in HV$ with $Hh(v)=x$ and $fh=gh$. Since $H$ is homological, any distinguished triangle
 \[
 \xymatrix{V \ar[r]^-{h} & X \ar[r]^-{f-g} & Y \ar[r] & V[1] } 
 \]
 in $\ca{T}$ yields such a pair $(V,v)$.
 \end{proof}
 
 It follows that the collection of additive subfunctors $i \colon S \rightarrow \ca{T}(-,X)$ with the property that $i \ten{\ca{T}} H$ is an isomorphism form an \emph{additive} Grothendieck topology $\tau$ on $\ca{T}$ in the sense of \cite[Definition~1.2]{BORCEUX_QUINTEIRO} (this is a many-object version of a Gabriel topology on a ring). We write $\sh_H(\ca{T})$ for the corresponding full subcategory of additive sheaves and
 \[
 L \colon [\ca{T}^{\op}, \Ab] \rightarrow \sh_H(\ca{T})
 \]
 for the exact reflector (this is the associated sheaf functor \cite[Theorem~4.4]{BORCEUX_QUINTEIRO}). Note that this is given by the usual plus construction, applied twice; the axioms for an additive Grothendieck topology ensure that the plus construction of an additive functor is again addtitive. It is exact since it is constructed using filtered colimits.

 For a general homological functor, we do not expect that this topology has good finiteness properties. As Pstr{\k a}gowski observed in \cite{PSTRAGOWSKI}, there is another natural topology on $\ca{T}$ induced by $H$ (see \cite[ \S 3.3]{PSTRAGOWSKI}).
 
 \begin{dfn}
 A morphism $p \colon X \rightarrow Y$ in $\ca{T}$ is called an \emph{$H$-epimorphism} if $Hp$ is an epimorphism.
 \end{dfn}
 
 The following result can be found in \cite[Lemma~3.19]{PSTRAGOWSKI}.
 
 \begin{prop}\label{prop:pstragowski}
 Let $\ca{T}$ be rigid $\otimes$-triangulated and let $H \colon \ca{T} \rightarrow (A,\varepsilon)\mbox{-}\Vect_K$ be a K{\"u}nneth functor. If
 \[
  \xymatrix{A \ar[r]^{f} & B \ar[r]^-{p} & C \ar[r] & A[1]}
 \]
is a distinguished triangle in $\ca{T}$ such that $p \colon B \rightarrow C$ is an $H$-epimorphism, then the sequence
 \[
 \xymatrix{0 \ar[r] & HA \ar[r]^-{f} & HB \ar[r]^-{p} & HC \ar[r] & 0}
 \]
 of $(A,\varepsilon)$-graded vector spaces is exact. In particular, $H$ sends homotopy pullbacks along $H$-epimorphisms to pullback diagrams, so $H$-epimorphisms are stable under homotopy pullback.
 \end{prop}

\begin{proof}
 The natural isomorphisms $(-)[1] \cong -\otimes S^1$ and $H(-\otimes S^1) \cong H(-) \otimes HS^1$ imply that $H(p[-1])$ is an epimorphism as well. Since $H$ is homological, we find from the exact sequence
 \[
 \xymatrix@C=40pt{H(B[-1]) \ar[r]^{Hp[-1]} & \ar[r] H(C[-1]) \ar[r]^-{Hq[-1]} & HA \ar[r]^{Hf}  & HB \ar[r]^{Hp} & HC }
 \]
 that $H(q[-1])=0$, from which the desired exactness follows.
 
 Homotopy pullbacks of $p \colon A \rightarrow B$ along $g \colon B^{\prime} \rightarrow B$ are defined by extending $(\begin{smallmatrix} p & -g \end{smallmatrix}) \colon A \oplus B^{\prime} \rightarrow B$ to a distinguished triangle, so the second claim follows from the first since $H(\begin{smallmatrix} p & -g \end{smallmatrix})$ is clearly an epimorphism. The final observation follows from the fact that epimorphisms are stable under pullback in any abelian category.
\end{proof}

 The $H$-epimorphisms thus constitute a \emph{singleton} Grothendieck coverage on $\ca{T}$. We can use this to generate an additive Grothendieck topology $\tau^P_H$, which we call the Pstr{\k a}gowski-topology, as follows: a subobject $i \colon S \rightarrow \ca{T}(-,X)$ lies in $\tau_H^P(X)$ if and only if there exists an $H$-epimorphism $p \colon X \rightarrow Y$ such that the morphism $\ca{T}(-,p)$ factors through $i$ (see for example \cite[ Appendix~A]{SCHAEPPI_COLIMITS}, though it is easier to see this here since $H$-epimorphisms are closed under finite compositions). We write
 \[
 \Sigma_H \defl \{ \xymatrix{F \ar[r]^-f & X \ar[r]^{p} & Y} \}
 \]
 for the set of sequences in $\ca{T}$ such that $p$ is an $H$-epimorphism and there exists $h \colon Y \rightarrow F[1]$ such that the triangle
 \[
 \xymatrix{F \ar[r]^-f & X \ar[r]^-p & Y \ar[r]^h & F[1]}
 \]
 is distinguished. The next result is key: it establishes the desired finiteness properties of $\sh_H(\ca{T})$.   
   
 \begin{prop}\label{prop:sheaf_characterization}
  Let $\ca{T}$ be rigid $\otimes$-triangulated and let $H \colon \ca{T} \rightarrow (A,\varepsilon)\mbox{-}\Vect_K$ be a K{\"u}nneth functor. Then the additive Grothendieck topologies $\tau_H$ and $\tau_H^P$ coincide. Moreover, an additive presheaf $G \colon \ca{T}^{\op} \rightarrow \Ab$ is a sheaf for $\tau_H$ if and only if it sends every sequence in $\Sigma_H$ to a left exact sequence
  \[
  \xymatrix{ 0 \ar[r] & GY \ar[r]^-{Gp} & GX \ar[r]^-{Gf} & GF }
  \]
  of abelian groups.
 \end{prop}
 
 \begin{proof}
 Let $j \colon S \rightarrow \ca{T}(-,X)$ be an element of $\tau_H(X)$, that is, a monomorphism such that $j \ten{\ca{T}} H$ is an isomorphism. We need to check that there exists an $H$-epimorphism $p \colon Y \rightarrow X$ such that $\ca{T}(-,p)$ factors through $i$. There exists a set $I$, objects $X_i \in \ca{T}$ and morphisms $f_i \colon \ca{T}(-, X_i) \rightarrow S$ such that the induced morphism $(f_i)_{i \in I} \colon \bigoplus_{i \in I} \ca{T}(-,X_i) \rightarrow S$ is an epimorphism in $[\ca{T}^{\op},Ab]$. Thus $j \circ (f_i)_{i \in I}$ is sent to an epimorphism by $-\ten{\ca{T}} H$. Note that $jf_i$ is of the form $\ca{T}(-,p_i)$ for a unique $p_i \colon X_i \rightarrow X$, and the natural isomorphism $\ca{T}(-,Y) \ten{\ca{T}} H \cong HY$ implies that the morphisms $H(p_i) \colon HX_i \rightarrow HX$ are jointly epimorphic. 
 
 Since $HX$ sends each object in $\ca{T}$ to an object with a dual, that is, a graded vector space of finite total dimension, it follows that there exists a finite subset $I_0 \subseteq I$ such that $(Hp_i)_{i \in I_0} \colon H(\bigoplus_{i \in I_0} X_i) \rightarrow HX$ is still an epimorphism. We have thus found our desired $H$-epimorphism $p=(p_i)_{i \in I_0} \colon \bigoplus_{i \in I_0} X_i \rightarrow X$ such that $\ca{T}(-,p)$ factors through $j$. We have shown that $\tau_H \subseteq \tau_H^P$.
 
 Conversely, if $j \colon S \rightarrow \ca{T}(-,X)$ is any morphism and $p \colon Y \rightarrow X$ is an $H$-epimorphism with $\ca{T}(-,p)=jq$ for some $q$, then $jq \ten{\ca{T}} H= \ca{T}(-,p) \ten{\ca{T}} H \cong Hp$ is an epimorphism. Thus $j \ten{\ca{T}} H$ is an epimorphism, but it is also a monomorphism since $- \ten{\ca{T}} H$ is exact by Proposition~\ref{prop:exactness}. Thus $j \in \tau_H(X)$, which concludes the proof that $\tau_H=\tau_H^P$.
 
 Since the Pstr{\k a}gowski-topology $\tau_H^P$ is the additive Grothendieck topology associated to the singleton Grothendieck coverage of $H$-epimorphisms, it follows from \cite[Proposition~A.2.5]{SCHAEPPI_COLIMITS} that $G$ is a sheaf for $\tau_H^P$ if and only if the functor $[\ca{T}^{\op},\Ab]$ sends the sequence sequence
 \begin{equation}\label{eqn:sheaf}
\xymatrix{ \bigoplus_{ \{f \colon Z \rightarrow X \vert pf=0\} } \ca{T}(-,Z)  \ar[r] & \ca{T}(-,X) \ar[r]^{\ca{T}(-,p)} & \ca{T}(-,Y)} \tag{$\star$}
  \end{equation}
 to a right exact sequence for every $H$-epimorphism $p$. Choose any distinguished triangle
 \[
 \xymatrix{F \ar[r]^-f & X \ar[r]^p & Y \ar[r] & F[1]}
 \]
 in $\ca{T}$. Since $F$ is a weak kernel in $\ca{T}$, the left morphism in ~\eqref{eqn:sheaf} factors through $\ca{T}(-,f) \colon \ca{T}(-,F) \rightarrow \ca{T}(-,X)$ via a \emph{split} epimorphism (since $f \colon F\rightarrow X$ is itself responsible for one of the direct summands). Thus $G$ is a sheaf for $\tau_H^P$ if and only if $[\ca{T}^{\op},\Ab](-,G)$ sends the sequence
 \[
 \xymatrix{ \ca{T}(-,F) \ar[r]^-{\ca{T}(-,f)} & \ca{T}(-,X) \ar[r]^-{\ca{T}(-,p)} & \ca{T}(-, Y ) } 
 \]
 to a right exact sequence for all sequences in $\Sigma_H$, so the second claim follows by Yoneda.
 \end{proof}

\begin{prop}\label{prop:faithful_exact}
  Let $\ca{T}$ be rigid $\otimes$-triangulated and let $H \colon \ca{T} \rightarrow (A,\varepsilon)\mbox{-}\Vect_K$ be a K{\"u}nneth functor. Then the functor
  \[
  \Hom_{\ca{T}}(F,-) \colon (A,\varepsilon)\mbox{-}\Vect_K \rightarrow [\ca{T}^{\op},\Ab]
  \]
 factors through the category $\sh_H{\ca{T}}$, so the restriction of $-\ten{\ca{T}} H$ to $\sh_H(\ca{T})$ gives a left adjoint to $\Hom_{\ca{T}}(F,-)$. This restriction is both exact and faithful.
\end{prop}

\begin{proof}
 By Proposition~\ref{prop:sheaf_characterization}, to check that $\Hom_{\ca{T}}(F,-)$ lands in sheaves, we need to check that $(A,\varepsilon)\mbox{-}\Vect_K(F-,V)$ sends sequences in $\Sigma_H$ to left exact sequences for each graded vector space $V$. Equivalently, we need to show that $F$ sends sequences in $\Sigma_H$ to \emph{right} exact sequences, which was proved in Proposition~\ref{prop:pstragowski}.
 
 It follows that the restriction of $-\ten{\ca{T}} H$ is a left adjoint, so it is right exact. It is also left exact since the inclusion of sheaves in presheaves is left exact and the functor $- \ten{\ca{T}} H \colon [\ca{T}^{\op},Ab] \rightarrow (A,\varepsilon)\mbox{-}\Vect_K$ is exact by Proposition~\ref{prop:exactness}.
 
 It remains to check that the restriction of $-\ten{\ca{T}} H$ to $\sh_H(\ca{T})$ is faithful. It clearly suffices to show the following: for any $G \in [\ca{T}^{\op},\Ab]$ such that $G \ten{\ca{T}} H \cong 0$, we also have $LG \cong 0$. Since the morphisms $f \colon \ca{T}(-,X) \rightarrow G$ are jointly epimorphic and $L$ is a left adjoint, it suffices to check that $Lf = 0$.
 
 Let $p \colon \ca{T}(-,X) \rightarrow G_0$ be the image of $f$ and let $j \colon S \rightarrow \ca{T}(-,X)$ be the kernel of $p$. Since $- \ten{\ca{T}} H$ is exact, it follows that $j \ten{\ca{T}} H$ is an isomorphism, so $j \in \tau_H(X)$. Since $\sh_H(\ca{T})$ is the category of sheaves for $\tau_H$, this implies that $Lj$ is an isomorphism. It follows that both $Lp=0$ and $Lf=0$, so $LG\cong 0$ since $f$ was an arbitrary morphism. 
\end{proof}

 Recall that an object $C$ of a category $\ca{C}$ is called \emph{finitely presentable} if $\ca{C}(C,-)$ preserves filtered colimits, and that $\ca{C}$ is called \emph{locally finitely presentable} if it has a strong generating set consisting of finitely presentable objects.
 
\begin{cor}\label{cor:lfp}
 Let $\ca{T}$ be rigid $\otimes$-triangulated and let $H \colon \ca{T} \rightarrow (A,\varepsilon)\mbox{-}\Vect_K$ be a K{\"u}nneth functor. Then the category $\Sh_{H}(\ca{T})$ is a locally finitely presentable abelian category. The left adjoint \[
 - \ten{\ca{T}} H \colon \sh_{H}(\ca{T}) \rightarrow (A,\varepsilon)\mbox{-}\Vect_K
\]
is comonadic and its right adjoint $\Hom_{\ca{T}}(F,-)$ is cocontinuous. Thus $\sh_{H}(\ca{T})$ is equivalent to the category of comodules of an $(A,\varepsilon)$-graded $K\mbox{-}K$-coalgebroid.
\end{cor}

\begin{proof}
 The category is abelian since the reflector $L \colon [\ca{T}^{\op},\Ab] \rightarrow  \sh_H{\ca{T}}$ is exact. From Proposition~\ref{prop:sheaf_characterization} it follows that $\sh_H(\ca{T})$ is closed under filtered colimits in $[\ca{T}^{\op},\Ab]$, that is, the right adjoint $U \colon \sh_H(\ca{T}) \rightarrow [\ca{T}^{\op},\Ab]$ of $L$ commutes with filtered colimits. Thus $L$ preserves finitely presentable objects, so the objects $L\bigl(\ca{T}(-,X) \bigr)$ form the desired generating set consisting of finitely presentable objects.
 
 The restriction of the left adjoint $- \ten{\ca{T}} H$ to sheaves is faithful and exact, so comonadic by Beck's theorem. To see that its right adjoint is cocontinuous, it thus suffices to check that the comonad is cocontinuous. But this is precisely the comonad associated to the adjunction
 \[
 -\ten{\ca{T}} H \dashv \Hom_{\ca{T}}(F,-) \colon [\ca{T}^{\op},\Ab] \rightarrow (A,\varepsilon)\mbox{-}\Vect_K \smash{\rlap{,}}
 \]
 so it suffices to check that the right adjoint, considered as a functor with target $[\ca{T}^{\op},\Ab]$, is cocontinuous. This is the case if and only if it commutes with filtered colimits (since every exact sequence in the domain is split). Both categories are locally finitely presentable, so we only need to check that the left adjoint preserves finitely presentable objects. This follows from the isomorphism $\ca{T}(-,X) \ten{\ca{T}} H \cong HX$ and the fact that $HX$ is finite dimensional (since $H$ preserves duals). Finally, a cocontinuous comonad on the category of $(A,\varepsilon)$-graded $K$-vector spaces is precisely an $(A,\varepsilon)$-graded $K\mbox{-}K$-coalgebroid.
\end{proof}

\subsection{Day convolution}

 Let $\ca{A}$ be a small $\Ab$-enriched symmetric monoidal category. From Day's thesis (see \cite{DAY}) we know that there is a symmetric monoidal closed structure on $[\ca{A}^{\op},\Ab]$ such that the Yoneda embedding is symmetric strong monoidal, with tensor product given by the Day convolution product. We denote the tensor product by $\ten{\mathrm{Day}}$ and the internal hom by $[-,-]_{\mathrm{Day}}$. Day's \emph{reflection} theorem gives general criteria for when a symmetric monoidal closed structure passes to a reflective subcategory. For example, if $\Sigma$ is a set of sequences $\{A_0 \rightarrow A_1 \rightarrow A_2\}$ in $\ca{A}$, then we can consider the full subcategory $\Lex_{\Sigma}(\ca{A}) \subseteq [\ca{A}^{\op},\Ab]$ of presheaves $G$ which send each sequence in $\Sigma$ to a left exact sequence
 \[
 \xymatrix{0 \ar[r] & GA_2 \ar[r] & GA_1 \ar[r] & GA_0}
 \]
 of abelian groups. By \cite[Theorem~6.11]{KELLY_BASIC}, this is a reflective subcategory (though of course not necessarily abelian without imposing further conditions on the set $\Sigma$). In this case, Day's reflection theorem tells us that there exists an essentially unique symmetric monoidal closed structure on $\Lex_{\Sigma}(\ca{A})$ if and only if for every $A \in \ca{A}$ and every $G \in \Lex_{\Sigma}(\ca{A})$, the internal hom $[\ca{A}(-,A),G]_{\mathrm{Day}}$ again lies in $\Lex_{\Sigma}(\ca{A})$. From the isomorphisms
 \begin{align*}
[\ca{A}(-,A),G]_{\mathrm{Day}}(B) & \cong [\ca{A}^{\op},\Ab]\bigl(\ca{A}(-,B),[\ca{A}(-,A),G]_{\mathrm{Day}}\bigr) \\
&\cong   [\ca{A}^{\op},\Ab]\bigl( \ca{A}(-,B \otimes A), G \bigr) \\
&\cong  G(B\otimes A)
 \end{align*}
 (Yoneda, definition of internal hom and strong monoidality of the Yoneda embedding, Yoneda), it follows that Day's reflection theorem is applicable if $\Sigma$ is closed under tensoring with $A$ for each $A \in \ca{A}$. Moreover, the resulting symmetric monoidal closed structure has a universal property. To state it, we need some notation. Let $\ca{C}$ be any cocomplete symmetric monoidal closed category. We write $\Fun_{\otimes,\Sigma}(\ca{A},\ca{C})$ for the category of symmetric strong monoidal additive functors which send sequences in $\Sigma$ to \emph{right} exact sequences in $\ca{C}$. Given a further cocomplete symmetric monoidal category $\ca{D}$, we write $\Fun_{\otimes, \mathrm{c}}(\ca{C},\ca{D})$ for the category of symmetric strong monoidal left adjoints (here $\mathrm{c}$ stands for cocontinuous, which is equivalent to being left adjoint for locally presentable categories). We similarly have categories $\Fun_{\mathrm{lax},\Sigma}(\ca{A},\ca{C})$ and $\Fun_{\mathrm{lax},c}(\ca{C},\ca{D})$ of \emph{lax} monoidal functors sending sequences in $\Sigma$ to right exact sequences, respectively the lax monoidal functors which are left adjoint.
 
  Let $Z \colon \ca{A} \rightarrow \Lex_{\Sigma}(\ca{A})$ be the composite $LY$. The following result is certainly known, though there does not seem to be a reference spelling it out at the required level of generality.

 \begin{thm}\label{thm:day_universal}
 Let $\ca{A}$ be a small additive symmetric monoidal category and $\Sigma$ a set of pairs of composable morphisms in $\ca{A}$ which is closed under tensoring with objects in $\ca{A}$. Then the functor $Z \colon \ca{A} \rightarrow \Lex_{\Sigma}(\ca{A})$ is the universal symmetric strong monoidal additive functor to a cocomplete symmetric monoidal closed additive category sending sequences in $\Sigma$ to cokernel diagrams: for every cocomplete symmetric monoidal closed category $\ca{C}$, the functor
 \[
 - \circ Z \colon \Fun_{\otimes, \mathrm{c}} \bigl(\Lex_{\Sigma}(\ca{A}), \ca{C} \bigr) \rightarrow \Fun_{\otimes, \Sigma}(\ca{A},\ca{C})
 \]
 is an equivalence of categories. Similarly, the functor
 \[
  - \circ Z \colon \Fun_{\mathrm{lax}, \mathrm{c}} \bigl(\Lex_{\Sigma}(\ca{A}), \ca{C} \bigr) \rightarrow \Fun_{\lax, \Sigma}(\ca{A},\ca{C})
 \]
 is an equivalence. In both cases, the inverse sends $G \colon \ca{A} \rightarrow \ca{C}$ to the restriction of $-\ten{\ca{A}} G$ to $\Lex_{\Sigma}(\ca{A})$.
 \end{thm}
 
 \begin{proof}
 From \cite[Theorem~5.1]{IM_KELLY}, we know that $-\circ Y$ induces an equivalence
 \[ 
 -\circ Y \colon \Fun_{\otimes, c}\bigl([\ca{A}^{\op},\Ab], \ca{C} \bigr) \simeq \Fun_{\otimes, \varnothing}(\ca{A},\ca{C})
 \]
 (and similarly for lax monoidal functors). Thus it suffices to check that precomposing with $L$ induces an equivalence between $\Fun_{\otimes, c}\bigl(\Lex_{\Sigma}(\ca{A}), \ca{C} \bigr)$ and the full subcategory of $\Fun_{\otimes, \mathrm{c}}([\ca{A}^{\op},\Ab],\ca{C})$ consisting of those $F$ which send sequences in $\Sigma$ to cokernel diagrams. This is the case if and only if the right adjoint of $F$ lands in $\Lex_{\Sigma}(\ca{A})$. Since this category is reflective, this happens if and only if $F \eta \colon F \Rightarrow FUL$ is an isomorphism (where $U$ denotes the right adjoint of $L$, that is, the inclusion). The lax monoidal structure of $U$ is built from instances of $\eta$ and isomorphisms coming from the strong monoidal structure of $L$, so the composite $FU$ is \emph{strong} monoidal and gives the desired inverse to $- \circ L$. The lax case follows more directly since the composite $FU$ of lax monoidal functors is always lax monoidal.
 
 The second statement is immediate from the proof of \cite[Theorem~5.1]{IM_KELLY}, which produces an inverse to $-\circ Y$ by endowing the left Kan extension $-\ten{\ca{A}} G$ of $G$ along the Yoneda embedding with the structure of a symmetric strong   monoidal functor (respectively lax monoidal if $G$ is only lax monoidal).
 \end{proof}
 
 \begin{rmk}
 Note that the proof of Theorem~\ref{thm:day_universal} works equally well for categories enriched in a complete and cocomplete symmetric monoidal closed category $\ca{V}$  and any set $\Sigma$ of cylinders closed under tensoring with objects in $\ca{A}$.
 \end{rmk}

 We now return to our K{\"u}nneth functor $H \colon \ca{T} \rightarrow (A,\varepsilon)\mbox{-}\Vect_K$. From Proposition~\ref{prop:sheaf_characterization}, we know that $\Sh_H(\ca{A})$ is precisely $\Lex_{\Sigma_H}(\ca{T})$, where $\Sigma_H$ denotes the set of homotopy fiber sequences of $H$-epimorphisms, that is, pairs of morphisms $f,p$ such that $Hp$ is an epimorphism and there exists a distinguished triangle $\xymatrix{ F \ar[r]^-{f} & X \ar[r]^-{p} \ar[r] & Y \ar[r] & F[1]}$ and $Hp$ in $\ca{T}$. Since $H$ is strong monoidal and $\ca{T}$ is rigid $\otimes$-triangulated, this set is clearly closed under tensoring with any object in $\ca{T}$. Thus we can apply Days reflection theorem to obtain a symmetric strong monoidal closed structure on $\Sh_{H}(\ca{T})$ such that $L \colon [\ca{T},\Ab] \rightarrow \sh_H(\ca{T})$ is symmetric strong monoidal. Note that the image of $\ca{A}$ under $LY$ generates $\sh_{H}(\ca{A})$ and consists of objects with duals since $\ca{T}$ is rigid. It follows immediately that finitely presentable objects are closed under tensor product.
 
 \begin{dfn}\label{dfn:m}
 Let $\ca{T}$ be rigid $\otimes$-triangulated and let $H \colon \ca{T} \rightarrow (A,\varepsilon)\mbox{-}\Vect_K$ be a K{\"u}nneth functor. We write $\ca{M}(H)$ for the full symmetric monoidal subcategory of $\sh_H(\ca{T})$ consisting of finitely presentable objects. We write
 \[
 [-] \colon \ca{T} \rightarrow \ca{M}(H)
 \]
 for the symmetric strong monoidal functor $LY$. We write $\bar{H}$ for the restriction of $-\ten{\ca{T}} H$ to $\ca{M}(H)$.
 \end{dfn}

\begin{thm}\label{thm:m}
 In the situation of Definition~\ref{dfn:m}, the category $\ca{M}(H)$ is a rigid abelian symmetric monoidal category (that is, every finitely presentable object of $\sh_H(\ca{T})$ has a dual). The objects $[X]$ for $X \in \ca{T}$ form a generator of $\ca{M}(H)$. The additive functor $\bar{H}$ is faithful, exact, and symmetric strong monoidal. In particular, the unit object of $\ca{M}(H)$ is simple. The diagram
 \[
 \xymatrix{ \ca{T} \ar[rr]^{[-]}  \ar[rd]_H && \ca{M}(H) \ar[ld]^{\bar{H} } \\ & (A,\varepsilon)\mbox{-}\Vect_K }
 \]
 commutes up to symmetric monoidal isomorphism (the functor $[-]$ is in particular homological). Moreover, the inclusion of $\ca{M}(H)$ in $\sh_H(\ca{H})$ induces an equivalence 
 \[
 \Ind\bigl(\ca{M}(H)\bigr) \rightarrow \sh_H(\ca{T})
 \]
 of symmetric monoidal closed categories, and this category is equivalent to the category of comodules of an $(A,\varepsilon)$-graded $K \mbox{-} K$ Hopf algebroid.
\end{thm}

\begin{proof}
 Note that the equivalence $\Ind\bigl(\ca{M}(H)\bigr) \simeq \sh_H(\ca{T})$ is immediate from the definition of $\ca{M}(H)$ as full subcategory of finitely presentable objects in the locally finitely presentable category $\sh_{H}(\ca{T})$.

 The functor $- \ten{\ca{T}} H \colon \sh_{H}(\ca{T}) \rightarrow (A,\varepsilon)\mbox{-} \Vect_K$ is faithful and exact by Proposition~\ref{prop:faithful_exact}. Moreover, since $H$ sends sequences in $\Sigma_H$ to cokernel diagrams by Proposition~\ref{prop:pstragowski}, it follows from the universal property of Day convolution that the restriction $- \ten{\ca{T}} H$ to $\sh_{H}(\ca{T})$ has a symmetric monoidal structure with a symmetric monoidal natural isomorphism $H \cong \bigl( LY(-) \ten{\ca{T}} H \bigr)$ (see Theorem~\ref{thm:day_universal}). In particular, the unit object is sent to a one-dimensional vector space by a faithful exact functor, so it is simple. The restriction to finitely presentable objects in the target yields the desired triangle of functors. 
 
 Since the duals generate $\sh_H(\ca{T})$, the adjunction $-\ten{\ca{T}}H \dashv \mathrm{Hom}_{\ca{T}}(H,-)$ satisfies the projection formula \cite[Proposition~3.8]{SCHAEPPI_TENSOR}, so the induced comonad is Hopf monoidal. The last claim follows from Corollary~\ref{cor:lfp}. Finally, a comodule over a Hopf algebroid has a dual if and only if its underlying object has a dual, so the category $\ca{M}(H)$ is indeed indeed rigid and abelian.
\end{proof}

\begin{rmk}
 The crucial property of rigidity in the above theorem can also be proved more directly since $-\ten{\ca{T}} H$ preserves the internal hom-objects $[M,N]_{\mathrm{Day}}$ if $M$ is finitely presentable (this follows easily by writing $M$ as cokernel of a morphism between duals), so $-\ten{\ca{T}} H$ detects duals since it is conservative.
\end{rmk}

 If $\ca{M}$ is a rigid abelian symmetric monoidal category such that there exists a faithful and exact symmetric strong monoidal functor to $(A,\varepsilon)$-graded $K$-vector spaces a \emph{graded Tannakian category}, and we call such functors \emph{graded fiber functors}.
 
 \begin{cor}\label{cor:natural_iso}
 Suppose $\ca{T}$ is rigid tensor triangulated and we have two K{\"u}nneth functors
 \[
 H \colon \ca{T} \rightarrow (A,\varepsilon)\mbox{-} \Vect_K \quad \text{and} \quad  H^{\prime} \colon \ca{T} \rightarrow (A,\varepsilon)\mbox{-} \Vect_{K^{\prime}}
 \]
 such that there exists common field extension $K^{\prime \prime}$ of $K$ and $K^{\prime}$ and a natural isomorphism $K^{\prime\prime} \ten{K} H \cong K^{\prime \prime} \ten{K^{\prime}} H^{\prime}$ of functors to $A$-graded $K^{\prime \prime}$-vector spaces (not necessarily compatible with the symmetric monoidal structure). Then $\ca{M}(H)=\ca{M}(H^{\prime})$ and there exist fiber functors
\[  
 \bar{H} \colon \ca{M}(H) \rightarrow (A,\varepsilon)\mbox{-} \Vect_K \quad \text{and} \quad  \bar{H}^{\prime} \colon \ca{M}(H) \rightarrow (A,\varepsilon)\mbox{-} \Vect_{K^{\prime}}
\]
 and symmetric monoidal isomorphisms $H \cong  \bar{H}[-]$ and $H^{\prime} \cong  \bar{H}^{\prime}[-]$. 
\end{cor}

\begin{proof}
 It is clear from the condition that the notions of $H$-epimorphism and $H^{\prime}$-epimorphism coincide. Thus $\ca{M}(H)=\ca{M}(H^{\prime})$ and the claim follows from Theorem~\ref{thm:m}.
\end{proof}

 The construction can for example be applied to the mixed Weil cohomology theories defined by Cisinski and D\'{e}glise \cite{CISINSKI_DEGLISE}. Given a perfect field $k$, we write $\mathrm{DM}_{\mathrm{gm}}(k)$ for Voevodsky's triangulated category of mixed motives. 
 
 \begin{cor}\label{cor:mixed_weil}
  Let $k$ be a perfect field and let $E$ be a mixed Weil cohomology theory on smooth $k$-schemes with coefficient field $K$. Then the realization functor $R_{\ca{E}} \colon \mathrm{DM}_{\mathrm{gm}}(k) \rightarrow \mathrm{D}^{b}(K)$ of \cite[Theorem~2.7.14]{CISINSKI_DEGLISE} naturally factors through a universal graded-Tannakian category $\ca{M}(R_{\ca{E}})$.
 \end{cor}

\begin{proof}
 By \cite[Theorem~2.7.14]{CISINSKI_DEGLISE} the functor $R_{\ca{E}}$ is triangulated and symmetric strong monoidal, so it gives rise to a K\"{u}nneth functor after applying the (homological) equivalence $\mathrm{D}^{b}(K) \simeq (\mathbb{Z},\varepsilon)\mbox{-}\Vect^{\fd}_K$. The existence of $\ca{M}(R_{\ca{E}})$ thus follows from Theorem~\ref{thm:m}; its universal property is discussed in \S \ref{section:universal} below.
\end{proof}

\subsection{The universal property}\label{section:universal}

 Let $\ca{A}$, $\ca{B}$ be abelian symmetric monoidal categories such that the tensor product is right exact in each variable. We write $\Fun_{\otimes, \mathrm{rex}}(\ca{A},\ca{B})$ for the category of right exact symmetric strong monoidal functors between them.

 For a triangulated category $\ca{T}$, we say that two homological functors $H \colon \ca{T} \rightarrow \ca{A}$ and $H^{\prime} \colon \ca{T} \rightarrow \ca{B}$ \emph{have the same strength} if for every morphism $f\colon X \rightarrow Y$ in $\ca{T}$, we have $Hf=0$ if and only if $H^{\prime} f=0$.
 
 \begin{thm}\label{thm:universal_property}
 In the situation of Definition~\ref{dfn:m}, the functor $[-] \colon \ca{T} \rightarrow \ca{M}(H)$ classifies K{\"u}nneth functors of the same strength as $H$. More precisely, sending $G \in \Fun_{\otimes, \mathrm{rex}}\bigl( \ca{M}(H),\ca{A} \bigr)$ to $G[-] \colon \ca{T} \rightarrow \ca{A}$ gives an equivalence between $\Fun_{\otimes, \mathrm{rex}}\bigl( \ca{M}(H),\ca{A} \bigr)$ and the category of K{\"u}nneth functors $\ca{T} \rightarrow \ca{A}$ of the same strength as $H$ for each abelian category $\ca{A}$ with right exact symmetric monoidal structure. Moreover, each functor $G \in \Fun_{\otimes, \mathrm{rex}}\bigl( \ca{M}(H),\ca{A} \bigr)$ is faithful and exact.
 \end{thm}

\begin{proof}
 The last statement was proved by Deligne for symmetric strong monoidal functors with arbitrary abelian rigid domain with simple unit object, see \cite[Corollaire~2.10]{DELIGNE}. Thus $G[-]$ is homological of the same strength as $H$.
 
 Conversely, if $H^{\prime} \colon \ca{T} \rightarrow \ca{A}$ is a K{\"u}nneth functor, we can compose it with the inclusion in $\Ind(\ca{A})$ to get a K{\"u}nneth functor whose target is cocomplete and symmetric monoidal closed. If $H$ and $H^{\prime}$ have the same strength, then the notions of $H$-epimorphism and $H^{\prime}$-epimorphism coincide. Applying Proposition~\ref{prop:pstragowski} to $H^{\prime}$, we find that $H^{\prime}$ sends each sequence in $\Sigma_H$ to a cokernel diagram in $\Ind(\ca{A})$. The conclusion follows from the universal property of Day convolution (Theorem~\ref{thm:day_universal}) and the observation that each object in $\Ind(\ca{A})$ which has a dual lies in $\ca{A}$ (since the tensor unit in $\Ind(\ca{A})$ is finitely presentable).
\end{proof}

\begin{cor}\label{cor:derived}
 Let $\ca{T}=\mathrm{D}^{b}(\ca{K})$ be the bounded derived category of a Tannakian category $\ca{K}$ with fiber functor $w \colon \ca{K}\rightarrow \Vect_K$. Let $H \colon \mathrm{D}^{b}(\ca{K}) \rightarrow (\mathbb{Z},\varepsilon)\mbox{-}\Vect_K$ be the K\"{u}nneth functor induced by $w$. Then the composite
 \[
\xymatrix{ \ca{K} \ar[r]^-{\mathrm{incl}_0} & \mathrm{D}^{b}(\ca{K}) \ar[r]^-{[-]} & \ca{M}(H) }
 \]
 is exact, full, and faithful.
\end{cor}

\begin{proof}
 The diagram
 \[
 \xymatrix{ \mathrm{D}^{b}(\ca{K}) \ar[d]_{\mathrm{D}^{b}(w)} \ar[r]^-{H_{\ast}} & (\mathbb{Z},\varepsilon)\mbox{-}\mathrm{gr}(\ca{K}) \ar[d]^{w_{\ast}} \\ 
 \mathrm{D}^{b}(\Vect_K)  \ar[r]_-{H_{\ast}} &  (\mathbb{Z},\varepsilon)\mbox{-}\Vect_K }
 \]
 commutes up to symmetric monoidal isomorphism since $w$ is exact. Since $w_{\ast}$ is also conservative, it follows that the homological functor $H_{\ast} \colon D^{b}(\ca{K}) \rightarrow (\mathbb{Z},\varepsilon)\mbox{-}\mathrm{gr}(\ca{K})$ is symmetric strong monoidal and of the same strength as $H$. From the universal property we get the factorization depicted on the right below
 \[
 \xymatrix@!C=20pt{ & \mathrm{D}^{b}(\ca{K}) \ar[rd]^(0.6){H_{\ast}} \\
  \ca{K} \ar[ru]^{\mathrm{incl}_0} \ar[rr]_-{\mathrm{incl}_0} && (\mathbb{Z},\varepsilon)\mbox{-}\mathrm{gr}(\ca{K})} \quad
  \xymatrix@!C=20pt{  \mathrm{D}^{b}(\ca{K}) \ar[rd]_(0.4){H_{\ast}} \ar[rr]^-{[-]} && \ca{M}(H)  \ar[ld]^(0.4){G} \\ & (\mathbb{Z},\varepsilon)\mbox{-}\mathrm{gr}(\ca{K})}
 \]
 and by combining this with the triangle on the left above, we find that the composite of $G$ and $[-] \circ \mathrm{incl}_0$ is naturally isomorphic to the exact, full, and faithful functor $\mathrm{incl}_0$. Since $G$ is exact and faitfhul, it follows that $[-] \circ \mathrm{incl}_0$ is exact. It also follows that $G$ is full on the image of $[-] \circ \mathrm{incl}_0$, so $G$ induces a bijection on the relevant hom-sets. This implies the remaining claim that $[-] \circ \mathrm{incl}_0$ is full.
\end{proof}

\begin{rmk} 
The above corollary suggests that one can use the category $\ca{M}(H)$ to construct a candidate for the Tannakian category coming from a hypothetical motivic $t$-structure by considering the smallest graded-Tannakian subcategory of $\ca{M}(H)$ generated by the objects which ``should'' lie in the heart of the $t$-structure.
\end{rmk}

 \begin{rmk}
 From Theorem~\ref{thm:day_universal} we also get a corresponding universal property for symmetric lax monoidal functors, though in this case, the functor need not preserve finitely presentable objects (since it might not preserve objects with duals). Nevertheless, we get an induced symmetric lax monoidal left adjoint on ind-objects if the original homological functor sends sequences in $\Sigma_H$ to cokernel diagrams by the universal property.
 \end{rmk}

\section{Applications to Grothendieck's categories of motives}\label{section:applications}

\subsection{The basic Tannakian factorization}

 Let $\ca{A}$ be a small additive rigid symmetric monoidal category with finite direct sums. Then the category of $\Chb(\ca{A})$ of bounded chain complexes in $\ca{A}$ is a rigid symmetric monoidal differential graded category. This means that the tensor product is a differential graded functor, so the homotopy category $\Kb(\ca{A})$ is a rigid $\otimes$-triangulated category.
 
 Now suppose that $\ca{K}$ is a rigid symmetric monoidal category which is also semi-simple abelian, for example, a category of graded vector spaces of finite total dimension. In this case, taking homology gives a symmetric monoidal equivalence $H_{\ast} \colon \Kb(\ca{K}) \rightarrow (\mathbb{Z},\varepsilon)\mbox{-}\mathrm{gr}(\ca{K})$ where $\varepsilon(n)=(-1)^{n}$. The equivalence is given by taking homology, so it turns triangulated functors into homological functors.
 
 If $H \colon \ca{A} \rightarrow \ca{K}$ is a symmetric strong monoidal additive functor, we get a triangulated functor $\Kb(H) \colon \Kb(\ca{A}) \rightarrow \Kb(\ca{K})$ which is symmetric strong monoidal. The resulting composite
 \[
\xymatrix{ \Kb(\ca{A}) \ar[r]^-{\Kb(H)} & \Kb(\ca{\ca{K}}) \ar[r]^-{H_{\ast}}_-{\simeq} & (\mathbb{Z},\varepsilon)\mbox{-}\mathrm{gr}(\ca{K}) }
 \]
 is thus a K{\"u}nneth functor. If $\ca{K}$ is itself the category of $(\mathbb{Z},\varepsilon)$-graded vector spaces, then we can apply Theorem~\ref{thm:m} and we obtain a diagram
 \[
\xymatrix@C=30pt{ \ca{A} \ar[r]^-{\mathrm{incl_0}} \ar[d]_H & \Kb(\ca{A}) \ar[d]_{\Kb(H)} \ar[r]^{[-]} & \ca{M}\bigl(H^{\ast}\Kb(H)\bigr) \ar[d]^{ \overline{H_{\ast}\Kb(H)}} \\
(\mathbb{Z},\varepsilon)\mbox{-}\Vect_K \ar[r]_-{\mathrm{incl_0}} &
\Kb\bigl((\mathbb{Z},\varepsilon)\mbox{-}\Vect_K\bigr) \ar[r]_-{H_{\ast}}^-{\simeq} & (\mathbb{Z},\varepsilon) \mbox{-}\mathrm{gr}\bigl((\mathbb{Z},\varepsilon)\mbox{-}\Vect_K \bigr)}
 \]
 which commutes up to symmetric monoidal isomorphism, where $\ca{M}\bigl(H_{\ast}\Kb(H)\bigr)$ is graded-Tannakian and $\overline{H_{\ast}\Kb(H)}$ is a graded fiber functor.

 We write $\ca{M}_H$ for the smallest full subcategory of $\ca{M}\bigl(H^{\ast}\Kb(H)\bigr)$ which contains the image of $[-] \circ \mathrm{incl}_0$ and is closed under finite direct sums, kernels, cokernels, finite tensor products and duals. We simply write $[-] \colon\ca{A} \rightarrow \ca{M}_H$ for the composite $[-] \circ \mathrm{incl}_0$.
 
  By commutativity of the above diagram (up to isomorphism), the restriction of $\overline{H_{\ast} \Kb(H)}$ to $\ca{M}_H$ factors through $H^{\ast} \circ \mathrm{incl}_0$ (since graded $K$-vector spaces concentrated in degree $\{0\} \times \mathbb{Z}$ are stable under the required constructions). We denote this functor by $\bar{H} \colon \ca{M}_H \rightarrow (\mathbb{Z},\varepsilon)\mbox{-}\Vect_K$. By construction, there exists a natural symmetric monoidal isomorphism $\bar{H}[-]\cong H$.

 \begin{dfn}\label{dfn:tf}
 Let $\ca{A}$ be a small additive rigid symmetric monoidal category with finite direct sums. Let $K$ be a field and let $H \colon \ca{A} \rightarrow (\mathbb{Z},\varepsilon)\mbox{-}\Vect_K$ be a symmetric strong monoidal additive functor. The factorization
 \[
 \xymatrix@C=10pt{ \ca{A} \ar[rd]_(0.45){H} \ar[rr]^-{[-]} & \ar@{}[d]|(0.4){\cong}& \ar[ld]^(0.45){\bar{H}} \ca{M}_H  \\ 
 & (\mathbb{Z},\varepsilon)\mbox{-}\Vect_K }
 \]
 of symmetric strong monoidal functors constructed above is called the \emph{basic Tannakian factorization} of $H$. 
 \end{dfn}

\begin{thm}\label{thm:tannakian_factorization}
 The basic Tannakian factorization has the following properties.
 \begin{enumerate}
 \item[(i)] The category $\ca{M}_H$ is graded-Tannakian and $\bar{H}$ is a graded fiber functor.
 
 \item[(ii)] Let $H^{\prime} \colon \ca{A} \rightarrow (\mathbb{Z},\varepsilon)\mbox{-}\Vect_{K^{\prime}} $ be a symmetric strong monoidal additive functor and suppose there exists a common field extension $K^{\prime \prime}$ of $K$ and $K^{\prime}$ such that the diagram
 \[
\xymatrix{\ca{A} \ar@{}[rd]|{\cong} \ar[r]^-{H} \ar[d]_{H^{\prime}} & (\mathbb{Z},\varepsilon)\mbox{-}\Vect_{K} \ar[d]^{K^{\prime\prime} \ten{K} -} \\  (\mathbb{Z},\varepsilon)\mbox{-}\Vect_{K^{\prime}} \ar[r]_-{K^{\prime} \ten{K} -} & (\mathbb{Z},\varepsilon)\mbox{-}\Vect_{K^{\prime\prime}} } 
 \]
 commutes up to natural isomorphism. Then $\ca{M}_H=\ca{M}_{H^{\prime}}$ and there exists a graded fiber functor $\bar{H^{\prime}} \colon \ca{M}_H \rightarrow (\mathbb{Z},\varepsilon)\mbox{-}\Vect_{K^{\prime}}$ such that $H^{\prime} \cong \bar{H^{\prime}}[-]$.
 \item[(iii)] Let $\ca{B}$ be an additive rigid symmetric monoidal category with finite direct sums and $H^{\prime} \colon \ca{B} \rightarrow (\mathbb{Z},\varepsilon)\mbox{-}\Vect_K$ a symmetric strong monoidal functor. If $F \colon \ca{A} \rightarrow \ca{B}$ is an additive symmetric strong monoidal functor and there exists a symmetric monoidal isomorphism $H^{\prime} F \cong H$, then there exists a symmetric strong monoidal faithful exact functor $\bar{F} \colon \ca{M}_H \rightarrow \ca{M}_{H^{\prime}}$, together with symmetric monoidal isomorphisms $\bar{H^{\prime}} \bar{F} \cong \bar{H}$ and $\bar{F}[-]\cong [F-]$.
 \item[(iv)] If the category $\ca{A}$ is semi-simple abelian and the functor $H$ is faithful, then $[-]\colon \ca{A} \rightarrow \ca{M}_H$ is an equivalence of categories.
 \end{enumerate}
\end{thm} 
 
 \begin{proof}
 Part~(i) is immediate from the construction. To see~(ii), note that $\Kb$ is 2-functorial, so the diagram
 \[
\xymatrix@C=50pt{ \Kb(\ca{A}) \ar@{}[rd]|{\cong} \ar[r]^-{\Kb(H)} \ar[d]_{\Kb(H^{\prime})} & \Kb\bigl( (\mathbb{Z},\varepsilon)\mbox{-}\Vect_{K} \bigr) \ar[d]^{\Kb(K^{\prime\prime} \ten{K} -)} \\  \Kb\bigl( (\mathbb{Z},\varepsilon)\mbox{-}\Vect_{K^{\prime}} \bigr) \ar[r]_-{\Kb(K^{\prime} \ten{K} -)} & \Kb\bigl( (\mathbb{Z},\varepsilon)\mbox{-}\Vect_{K^{\prime\prime}} \bigr) }
 \]
 commutes up to natural isomorphism. Since taking homology commutes with the exact functors $K^{\prime\prime} \ten{K} -$ and $K^{\prime\prime} \ten{K^{\prime}} -$, we are in the situation of Corollary~\ref{cor:natural_iso}, which proves Part~(ii).

 Part~(iii) is proved similarly. From the 2-functoriality of $\Kb$ we get a symmetric monoidal isomorphism $\Kb(H^{\prime}) \Kb(F) \cong \Kb(H)$, so $\Kb(F)$ sends $H_{\ast}\Kb(H)$-epimorphisms to $H_{\ast}\Kb(H)$-epimorphisms. Since $\Kb(F)$ is triangulated, $[-] \circ \Kb(F)$ is homological and we get a symmetric strong monoidal functor
 \[
 \widetilde{F} \colon \ca{M}\bigl(H_{\ast}\Kb(H)\bigr) \rightarrow \ca{M}\bigl(H_{\ast}\Kb(H^{\prime})\bigr)
 \]
 and a symmetric monoidal isomorphism $\widetilde{F}[-] \cong [-] \circ \Kb(F)$ by the universal property (see Theorem~\ref{thm:universal_property}). The essential uniqueness part of the universal property shows that there exists a symmetric monoidal isomorphism $\overline{H_{\ast}\Kb(H^{\prime})} \widetilde{F} \cong \overline{H_{\ast}\Kb(H)}$, so $\widetilde{F}$ is faithful and exact, and the claim follows by restricting $\widetilde{F}$ to the full subcategory $\ca{M}_H$.

 Finally, if $\ca{A}$ is semi-simple abelian, then $H$ is exact. Moreover, $\Kb(\ca{A}) \simeq (\mathbb{Z},\varepsilon)\mbox{-}\mathrm{gr}(\ca{A})$ is semi-simple as well and $\Kb(H) \cong (\mathbb{Z},\varepsilon)\mbox{-}\mathrm{gr}(H)$ is faithful and exact. Thus the functor
 \[
 H_{\ast} \Kb(H) \colon \Kb(\ca{A}) \rightarrow (\mathbb{Z},\varepsilon) \mbox{-}\mathrm{gr}\bigl((\mathbb{Z},\varepsilon)\mbox{-}\Vect_K \bigr)
 \]
 \emph{detects} epimorphisms. Since all epimorphisms in a semi-simple abelian category are split, this means that the $ H_{\ast} \Kb(H)$-epimorphisms in $\Kb(\ca{A})$ are precisely the split epimorphisms. Thus the sequences in $\Sigma_{H_{\ast} \Kb(H)}$ are the split exact sequences, so $\Sh_{H_{\ast} \Kb(H)}\bigl(\Kb(\ca{A}) \bigr)$ is equal to the presheaf category $[\Kb(\ca{A})^{\op},\Ab ]$ and $[-]=Y$ is the Yoneda embedding (see Proposition~\ref{prop:sheaf_characterization}). Since the image of the full and faithful functor $[-] \circ \mathrm{incl_0} \colon \ca{A} \rightarrow \ca{M}_H$ is already closed under the desired operations, the claim follows.
\end{proof}

 By unravelling the construction, we get the following way of constructing morphisms between objects in the image of $[-]$.
 
 \begin{dfn}\label{dfn:h}
 Let $A$, $B$ be objects of $\ca{A}$ and consider them as objects of $\Kb(\ca{A})$ concentrated in degree zero. Consider a small diagram $C\colon \ca{I} \rightarrow \Kb(\ca{A})$, a cone $(f_i \colon A \rightarrow C_i)$ and a cone $(w_i \colon B \rightarrow C_i)$. The morphisms
 \[
 \xymatrix@C=45pt{HA \ar[r]^-{\cong} & H_\ast \Kb(H)(A) \ar[r]^{H_\ast \Kb(H)(f_i)} & H_\ast \Kb(H)(C_i) }
 \]
 of graded $K$-vector spaces induce morphisms
 \[
\widetilde{Hf} \colon HA \rightarrow \colim_{\ca{I}}\bigl(H_\ast \Kb(H)(C_i) \bigr) \quad \text{and} \quad \widetilde{Hw} \colon HB \rightarrow \colim_{\ca{I}}\bigl(H_\ast \Kb(H)(C_i) \bigr) 
 \]
 and we call the pair $\bigl((f_i),(w_i)\bigr)$ a \emph{homological morphism} from $A$ to $B$ if $\widetilde{Hw}$ is an isomorphism. We denote homological morphisms by $(f,w) \colon A \rightsquigarrow B$, leaving the diagram implicit. We write $\mathrm{H}(A,B)$ for the set of homological morphisms from $A$ to $B$.
 \end{dfn}

\begin{prop}\label{prop:homological_morphism}
 There is a surjective function $[-] \colon H(A,B) \rightarrow \ca{M}_H([A],[B])$ such that for each homolgical morphism $(f,w) \colon A \rightsquigarrow B$ as in Definition~\ref{dfn:h},  the diagram
 \[
 \xymatrix{HA \ar[d]_\cong \ar[r]^-{\widetilde{Hf}} & \colim_{i \in \ca{I}}\bigl(H_\ast \Kb(H)(C_i) \bigr) \ar[r]^-{\widetilde{Hw}^{-1}}  & HB \ar[d]^{\cong}\\ 
 \bar{H}[A] \ar[rr]_-{\bar{H}[(f,w)]} & &\bar{H} B}
 \]
 of graded $K$-vector spaces is commutative.
\end{prop}

\begin{proof}
 The cones $A \rightarrow C_i$ and $w_i \colon B \rightarrow C_i$ induce morphisms $\tilde{f} \colon YA \rightarrow \colim_i YC_i$ and $ \tilde{w} \colon YB \rightarrow \colim_{\ca{I}} YC_i$ of presheaves. Moreover, $L(\tilde{w})$ is invertible by assumption on $w$, and the morphism $[f,w] \defl L(\tilde{w})^{-1} L(f)$ makes the above diagram commutative.
 
 To see that this assignment is surjective, note that to give a morphism $[A]=LYA \rightarrow LYB=[B]$ amounts to giving a morphism $f \colon YA \rightarrow LYB$ by universal property of the associated sheaf functor. Since every presheaf is a colimit of representable presheaves, we have $LYB \cong \colim_{\ca{I}} YC_i$ for some diagram $C \colon \ca{I} \rightarrow \Kb(H)$. Since both $YA$ and $YB$ are projective, we can pick lifts $f_i$ of $f$ respectively $w_i$ of $w$ along the epimorphism $\oplus_{i \in \ca{I}} YC_i \rightarrow \colim_{\ca{I}} YC_i$ to obtain the desired homological morphism $(f,w) \colon A \rightsquigarrow B$.
\end{proof}

 We now apply the basic Tannakian factorization in the case where $\ca{A}$ is a category of Chow motives and $H$ is a Weil cohomology theory. 
 
\subsection{Tannakian categories associated to Weil cohomology theories}

 Let $k$ be a field and $\mathrm{SmProj}_k$ the category of smooth projective varieties over $k$. Let $K$ be a field of characteristic zero and $H^{\ast} \colon \mathrm{SmProj}_k \rightarrow (\mathbb{Z},\varepsilon)\mbox{-}\Vect_K$ a Weil cohomology theory. We write $\mathrm{Mot}_H(k)$ for the category of Chow motives modulo homological equivalence. The category $\mathrm{Mot}_H(k)$ is a rigid additive symmetric monoidal category and has finite direct sums. The Weil cohomology theory $H$ induces an additive symmetric strong monoidal functor $H \colon \mathrm{Mot}_H(k) \rightarrow (\mathbb{Z},\varepsilon)\mbox{-}\Vect_K$, which is faithful by our choice of equivalence relation. Let
 \[
 \xymatrix@C=10pt{ \mathrm{Mot_H}(k) \ar[rd]_(0.45){H} \ar[rr]^-{[-]} & \ar@{}[d]|(0.4){\cong}& \ar[ld]^(0.45){\bar{H}} \ca{M}_H(k)  \\ 
 & (\mathbb{Z},\varepsilon)\mbox{-}\Vect_K }
 \]
 be the basic Tannakian factorization (see Definition~\ref{dfn:tf}). We call homological morphisms in these categories \emph{homological cycles}. These categories exist unconditionally, and we get the following result.
 
 \begin{thm}\label{thm:properties}
 The category $\ca{M}_{H}$ and the functors $[-]$, $\bar{H}$ have the following properties.
 \begin{enumerate}
 \item[(i)]  If the standard conjecture D holds for $H$, then $[-] \colon \mathrm{Mot_H}(k) \rightarrow \ca{M}_H(k)$ is an equivalence.
  \item[(ii)] If $H^{\prime} \colon \mathrm{SmProj}_k \rightarrow (\mathbb{Z},\varepsilon)\mbox{-} \Vect_{K^{\prime}}$ is another Weil cohomology theory and there exists a common field extension $K^{\prime \prime}$ of $K$ and $K^{\prime}$ such that the diagram
  \[
  \xymatrix{\mathrm{Mot}_H(k) \ar@{}[rd]|{\cong} \ar[r]^-{H} \ar[d]_{H^{\prime}} & (\mathbb{Z},\varepsilon)\mbox{-}\Vect_{K} \ar[d]^{K^{\prime\prime} \ten{K} -} \\  (\mathbb{Z},\varepsilon)\mbox{-}\Vect_{K^{\prime}} \ar[r]_-{K^{\prime} \ten{K} -} & (\mathbb{Z},\varepsilon)\mbox{-}\Vect_{K^{\prime\prime}} } 
  \]
  commutes up to natural isomorphism (which amounts to compatibility with transfers, cup products, and the respective cycle maps), then there exists a graded fiber functor $\bar{H^{\prime}} \colon \ca{M}_H(k) \rightarrow (\mathbb{Z},\varepsilon)\mbox{-}\Vect_{K^{\prime}}$ and a symmetric monoidal isomorphism $H^{\prime} \cong \bar{H^{\prime}}[-]$.
 \item[(iii)] If $k$ has characteristic zero, $H$ is classical (de Rham, $\ell$-adic, respecitvely Betti cohomology if $k \subseteq \mathbb{C}$), and $\ca{M}_k$ is Andr{\'e}'s Tannakian category defined via motivated cycles \cite[\S 4]{ANDRE}, then there exists a (non symmetric) strong monoidal functor
  \[
  \ca{M}_H(k) \rightarrow \ca{M}_k
  \]
  which is faithful and exact and commutes with the respective realizations up to (non-symmetric) natural monoidal isomorphism.
\end{enumerate}   
 \end{thm}
 
 \begin{proof}
  From Jannsen's theorem \cite[Theorem~1]{JANNSEN}, we know that $\mathrm{Mot_H}(k)$ is semi-simple abelian if $H$-equivalence coincides with numerical equivalence (that is, if conjecture D holds for $H$). The first claim thus follows from Part~(iv) of Theorem~\ref{thm:tannakian_factorization}. 

 Part~(ii) is simply Part~(ii) of Theorem~\ref{thm:tannakian_factorization} for $\ca{A}=\mathrm{Mot}_H(k)$.
  
  It remains to show that Part~(iii) holds. Note that $\ca{M}_k$ is obtained by twisting the signs on a category $\ca{M}^{\mathrm{tw}}_k$ which is graded-Tannakian \cite[\S 4.3]{ANDRE}. There exists an evident triangle of \emph{symmetric} strong monoidal functors
  \[
   \xymatrix@C=5pt{ \mathrm{Mot_H}(k) \ar[rd]_(0.45){H} \ar[rr]^-{h} & \ar@{}[d]|(0.4){\cong}& \ar[ld]^(0.45){\ca{H}} \ca{M}^{\mathrm{tw}}_k  \\ 
 & (\mathbb{Z},\varepsilon)\mbox{-}\Vect_K }
  \]
  which commutes up to symmetric monoidal isomorphism. From the functoriality of the basic Tannakian factorization (Part~(iii) of Theorem~\ref{thm:tannakian_factorization}), we thus get a diagram
  \[
   \xymatrix@C=5pt{ \mathrm{Mot_H}(k) \ar[d]_{[-]}  \ar[rr]^-{h} & & \ca{M}^{\mathrm{tw}}_k \ar[d]^{[-]} \\ 
 \ca{M}_{H}(k) \ar[rd]_(0.45){\bar{H}}  \ar[rr]^-{\bar{h}} && \ca{M}_{\ca{H}}  \ar[ld]^(0.45){\bar{\ca{H}}} \\
 & (\mathbb{Z},\varepsilon)\mbox{-}\Vect_K }  
  \]
  which commutes up to symmetric monoidal isomorphism. Moreover, since the category $\ca{M}^{\mathrm{tw}}_k$ is semi-simple abelian and $\ca{H}$ is faithful (\cite[\S 4.4]{ANDRE}), Part~(iv) of Theorem~\ref{thm:tannakian_factorization} tells us that the right vertical morphism is a symmetric monoidal equivalence. We get the desired functor by composing with its inverse and the (non-symmetric) monoidal isomorphism $\id \colon \ca{M}^{\mathrm{tw}}_k \rightarrow \ca{M}_k$.
 \end{proof}

 \subsection{Homological standard conjectures}\label{section:standard}
 
 The standard conjectures imply various structural properties for the category $\ca{M}_H(k)$. In fact, rather than asking for the existence of algebraic cycles, it suffices to  assume that certain homological cycles exist. To state the conjectures, we need to recall a few basic facts about $\mathrm{Mot}_H(k)$.

 The assignment which sends a smooth projective variety $X$ to $(X,\id,0)$ defines a contravariant symmetric strong monoidal functor $\mathrm{SmProj}_k^{\op} \rightarrow \mathrm{Mot_H}(k)$ such that $H(X,\id,0)\cong H^{\ast} X$. For this reason, we simply write $X$ for $(X,\id,0)$. 
 There is an invertible object $L \in \mathrm{Mot_H}(k)$ called the \emph{Lefschetz motive}. Its image $H(L)$ is a 1-dimensional vector space concentrated in degree $2$. We need the following facts about the category $\mathrm{Mot_H}(k)$:
\begin{enumerate}
 \item[(i)] The objects $X \otimes L^{n}$ where $X$ is an irreducible smooth projective variety and $n \in \mathbb{Z}$ generate $\mathrm{Mot_H}(k)$ up to finite direct sums and summands.
 \item[(ii)] For $X$ irreducible smooth projective of dimension $d$, there is an isomorphism $X^{\vee} \cong X \otimes L^{-d}$. \item[(iii)] Given an ample line bundle $\ca{L}$ on $X$, let $\xi=c_1(\ca{L}) \in H^{2} X$ denote its first Chern class. There is a morphism $\ell \colon X \rightarrow X \otimes L^{-1}$ in $\mathrm{Mot_H}(k)$ such that $H\ell$ induces the linear map $\xi \cdot - \colon H^{\ast}X \rightarrow H^{\ast+2}X$.
\end{enumerate}

 We can adapt Grothendieck's standard conjectures and the variations discussed in \cite{KLEIMAN} from $\mathrm{Mot}_H(k)$ to $\ca{M}_{H}(k)$ by replacing algebraic cycles with homolgical cycles. Since homological cycles are defined relative to a Weil cohomology theory, the resulting conjectures depend on the Weil cohomology theory as well.
 
 Fix an irreducible smooth projective variety $X$ of dimension $d$. The standard conjecture $hC(X,H)$ states that there are homological cycles $\pi_i \colon X \rightsquigarrow X$, $i=0, \ldots, 2d$ such that the $[\pi_i] \colon [X] \rightarrow [X]$ in $\ca{M}_H(k)$ are idempotent and they induce a direct sum decomposition $[X] \cong \oplus_{i=0}^{2d} X_i$ with $\bar{H} X_i$ is concentrated in degree $i$. We write $hC(H)$ for the conjecture that $hC(X,H)$ holds for all irreducible smooth projective varieties $X$.

 
 The standard conjecture $h\nu(X,H)$ (cf.\ \cite[Theorem~2.9]{KLEIMAN}) states that there are homological cycles $\nu_i \colon X \otimes L^{-(d-i)} \rightsquigarrow X$ such that $[\nu_i]\colon [X \otimes L^{-(d-i)}] \rightarrow [X]$ induces an isomorphism $H^{i}X \rightarrow H^{2d-i} X$ in degree $i$ for $0 \leq i <d$ (not necessarily the inverse of $H^{i} \ell^{(d-i)}$). We similarly write $h\nu(H)$ for the conjecture that $h\nu(X,H)$ holds for all $X$.

 Recall that the standard conjecture $B$ asserts that various operators such as $\Lambda \colon H^{\ast}X \Rightarrow H^{\ast} X$ defined using a fixed polarization coming from the ample line bundle $\ca{L}$ are induced by a morphism in $\mathrm{Mot}_H(X)$, that is, an algebraic cycle (see \cite[\S 1.4]{KLEIMAN} for details). The corresponding conjecture $hB(X,H)$ relaxes this to the requirement that these operators are induced by \emph{homological} cycles in $\ca{M}_H(k)$. By definition of the operators, $hB(X,H)$ implies $h\nu(X,H)$.
 
 Kleiman showed in \cite[Corrollary~2.14]{KLEIMAN} that, if the conjecture $B$ holds for all $X$ and both Lefschetz theorems hold for $H$, then $C$ holds universally as well. For the homological counterparts of the conjectures, there is also a converse to this statement. Recall that the \emph{weak Lefschetz theorem} holds for $H$ if for each smooth hyperplane section $j \colon Y \subseteq X$, the corresponding morphism $j^{\ast} \colon X \rightarrow Y$ in $\mathrm{Mot}_H(k)$ induces an isomorphism $H^{i}X \rightarrow H^{i}Y$ if $i \leq d-2$, and the induced morphism is injective if $i=d-1$. The \emph{strong Lefschetz theorem} asserts that the morphism $\ell^{(d-i)} \colon X \rightarrow X \otimes L^{-(d-i)}$ induces an isomorphism $H^{i} X \rightarrow H^{2d-i} X$ for all $i \leq d$.

 The geometric input we need to compare these conjectures is Bertini's theorem: for each $d$-dimensional irreducible smooth projective variety, there exists a smooth hyperplane section $Y \subseteq X$, so $Y$ has irreducible components of dimension $d-1$ (see \cite[\S 6]{JOUANOLOU} for $k$ infinite respectively \cite{POONEN} for $k$ finite; in the latter case, one might need to change the projective embedding to get the desired hyperplane). This allows us to argue by induction on $d$. For the base case, we need the fact that for a one-dimensional irreducible smooth projective variety $X$ (that is, a curve), there is a decomposition $X \cong X_0 \oplus X_1 \oplus X_2$ in $\mathrm{Mot}_H(k)$ such that $H^{\ast}X_i$ is concentrated in degree $i$. This follows from the well-known fact that the standard conjecture $C(X)$ holds for curves, see \cite[\S 10]{MANIN} and \cite[\S 3.3]{SCHOLL} for general ground fields (the splitting is obtained from a rational point over a suitable extension of the base field).
 
 \begin{thm}\label{thm:lefschetz}
 Let $H$ be a Weil cohomology theory for which the weak Lefschetz theorem holds. If $h\nu(H)$ holds universally, then $hC(H)$ holds universally. If the hard Lefschetz theorem also holds for $H$, then the universal conjectures $h\nu(H)$, $hB(H)$, and $hC(H)$ are all equivalent.
 \end{thm}
 
 \begin{proof}
 We first assume that $h\nu(X,H)$ holds for all $X$. We prove that $hC(X,H)$ holds by induction on the dimension of $X$. The base case is taken care of by the fact that $[X]\cong[X_0]\oplus [X_1] \oplus [X_2]$ if $X$ is a curve.
 
 Thus assume that $\mathrm{dim}(X) \geq 2$. By Bertini's theorem, we can find a smooth hyperplane section $Y \subseteq X$, and there exists a decomposition $[Y] \cong \bigoplus_{i=0}^{2d-2} Y_i$ such that $\bar{H}Y_i$ is concentrated in degree $i$ (by induction assumption and the fact disjoint unions of varieties give direct sums of motives). The inclusion $Y \subseteq X$ thus induces a morphism $q \colon [X] \rightarrow [Y]\cong \bigoplus_{i=0}^{2d-2} Y_i$ with components $q_i \colon [X] \rightarrow Y_i$. Let 
 \[
 \xymatrix{[X] \ar[r]^{q_{d-1}^{\prime}} & Y_{d-1}^{\prime} \ar[r] & Y_{d-1}  }
 \]
 denote the image factorization of $q_{d-1}$ in the abelian category $\ca{M}_H(k)$. The weak Lefschetz theorem implies that $\bar{H}^{i} q_i \colon \bar{H}^{i} [X] \rightarrow \bar{H}^{i} Y_i$ is an isomorphism for $i \leq d-2$ and that $\bar{H}^{d-1} q_{d-1}^{\prime} \colon \bar{H}^{i} [X] \rightarrow \bar{H}^{d-1} Y_{d-1}^{\prime}$ is a monomorphism, hence an isomorphism (recall that $\bar{H}$ is exact). For $i \leq d-2$, let $X_i=Y_i$ and $p_i=q_i$. For $i = d-1$ let $X_{d-1}=Y_{d-1}^{\prime}$ and $p_{d-1}=q_{d-1}^{\prime}$. 
 
 The morphism $p_0^{\vee} \colon X_0^{\vee} \rightarrow [X]^{\vee} \cong [X \otimes L^{-d}]$ induces an isomorphism in degree zero as well. Thus the composite
 \[
 \xymatrix{ X_0^{\vee} \ar[r]^-{p_0^{\vee}} &  [X \otimes L^{-d}] \ar[r]^-{[\nu_0]} & [X] \ar[r]^{p_0} & X_0 }
 \]
induces an isomorphism in degree zero. Since $\bar{H}$ is faithful and the domain and codomain are concentrated in degree zero, it follows that $X_0$ is a direct summand of $[X]$. Taking duals and shifting, we find that there is a corresponding summand $X_{2d}$ and thus a direct sum decomposition $[X] \cong X_0 \oplus X^{\prime} \oplus X_{2d}$. Using the composite of $[\nu_1]$ with the inclusion and projection, we obtain a morphism $X^{\prime} \otimes [L^{-(d-1)}] \rightarrow X^{\prime}$ which induces an isomorphism in degree one. Combining this with $p_1^{\vee} \otimes [L]$, we find that $X_1$ is a direct summand and we get $[X] \cong X_0 \oplus X_1\oplus X^{\prime \prime} \oplus X_{2d-1} \oplus X_d$. Proceeding inductively we are eventually left with $X_d$ in the middle dimension, so $C(X,H)$ holds.

 Now assume that $H$ also satisfies the hard Lefschetz theorem, that is, $H^{i} \ell^{(d-i)}$ is an isomorphism for all $0 \leq i<d$. Since $hB(H)$ implies $h\nu(H)$, it only remains to show that $hC(H)$ implies $hB(H)$. Thus we can assume that $[X] \cong \oplus_{i=0}^{2d-i} X_i$.  In order to check that the required operators are homological, it suffices to observe that the decomposition into primitive parts of $H^{\ast} X$ is the image of a corresponding decomposition in $\ca{M}_H(k)$. To see this, let $P^i \in \ca{M}_H(k)$ be the kernel of the composite
 \[
 \xymatrix{X_i \ar[r] & [X] \ar[r]^-{\ell^{d-i+1}} &[X \otimes L^{-(d-i+1)}] }
 \]
 and let $L^j P^{i-2j}$ be the image of $\ell^j \otimes [L^j] \colon P^{i-2j} \otimes [L^j] \rightarrow X_i$. Then exactness and faithfulness of $\bar{H}$ and the hard Lefschetz theorem imply that $X_i \cong \oplus_{j \geq \mathrm{max}(d-i,0)} L^jP^{i-2j}$. By construction, the image of this isomorphism is (up to canonical isomorphism) the primitive decomposition of $H^{\ast} X$. The operators $\Lambda$, ${}^{c}\Lambda$, $\ast$, and $p_j$ of \cite[\S 1.4]{KLEIMAN} can thus all be defined in the category $\ca{M}_H(k)$, so $hB(H)$ holds.  
 \end{proof}
 
 Note that conjecture $hC(H)$ implies that all objects in $\ca{M}_H(k)$ have a direct sum decomposition whose summands are pure (concentrated in a single degree). Given a family of additive categories $\ca{A}_i$, we write $\bigoplus_{i \in I} \ca{A}_i$ for the full subcategory of of $\prod_{i \in I} \ca{A}_i$ of objects $(A_i)_{i \in I}$ such that $A_i=0$ for all but finitely many $i \in I$.
 
 \begin{prop}\label{prop:pure}
 Suppose that $H$ is a Weil cohomology for which $hC(H)$ holds universally. Write $\ca{M}_i$ for the full subcategory of $\ca{M}_H(k)$ consisting of objects $M$ with $\bar{H}^{j}M=0$ unless $j=i$. Then taking direct sums induces an equivalence $\bigoplus_{i \in \mathbb{Z}} \ca{M}_i \rightarrow \ca{M}_{H}(k)$. 
 \end{prop}
 
 \begin{proof}
 Since $\bar{H}$ is faithful, there are no non-zero morphisms between $M \in \ca{M}_i$ and $N \in \ca{M}_j$ unless $i=j$, so the functor is full and faithful and its image is closed under taking kernels and cokernels.  It is clearly closed under direct sums, tensor products, duals, and under tensoring by arbitrary powers of $[L]$. By assumption, the image contains $[X]$ for all irreducible smooth projective varieties $X$, so it contains $[M]$ for all $M \in \mathrm{Mot}_H(k)$. The claim follows from the definition of the basic Tannakian factorization (see Definition~\ref{dfn:tf}).
 \end{proof}
 
\begin{thm}\label{thm:CB}
 If $k$ is algebraic over a finite field and $H$ is {\'e}tale cohomology for a prime different from $p=\mathrm{char}(k)$ or crystalline cohomology, then $hB(H)$ holds universally. In this case, the signs of the symmetry on $\ca{M}_H(k)$ can be twisted and we obtain a genuine Tannakian category $\ca{M}_{\text{{\'e}t}, \ell}(k) \defl \ca{M}_H(k)^{\mathrm{tw}}$ with fiber functor $\bar{H}$ landing in $\Rep(\mathbb{G}_m,\mathbb{Q}_{\ell})$, respectively a Tannakian category $\ca{M}_{\mathrm{crys},p}(k)$. Moreover, all motivated cycles are morphisms in the Tannakian categories $\ca{M}_{\text{{\'e}t}, \ell}(k)$.
\end{thm}

\begin{proof}
 Recall that Deligne proved the hard Lefschetz theorem for {\'e}tale cohomology in \cite[Th\'{e}or\`{e}me~4.1.1]{DELIGNE_WEIL2}. Katz and Messing showed how the solution of the Weil conjectures implies the standard conjecture C for the Weil cohomology theories in question over a finite field \cite[Theorem~2]{KATZ_MESSING}, so conjecture hC(H) holds for algebraic extensions of a finite field. The hard Lefschetz theorem holds for crystalline cohomology by \cite[Corollary~1.2)]{KATZ_MESSING}). Thus both $hC(H)$ and $hB(H)$ hold for $\ca{M}_H(k)$ by Theorem~\ref{thm:lefschetz}. 
 
 By Proposition~\ref{prop:pure}, we get a decomposition of the category as a ``direct sum,'' which makes it possible to change the signs of the symmetry isomorphism on both the domain and codomain of $\bar{H}$. Finally, the motivated cycles are obtained by closing algebraic cycles under the operator $\ast$  (see \cite[\S 2]{ANDRE}), so by $hB(H)$, they all lie in the image of the natural isomorphism $\bar{H}[X] \cong H^{\ast}X$.
 \end{proof}

 Note that this gives examples of homological cycles which are not known to be algebraic (the standard conjecture B is not known to hold universally for fields which are algebraic over finite fields). 
 
 The standard conjecture $D$ for $H$ is also equivalent to the statement: the category $\mathrm{Mot}_H(k)$ is semi-simple (see \cite[Theorem~1]{JANNSEN}). Thus the corresponding conjecture $hD(H)$ would be that $\ca{M}_H(k)$ is semi-simple.
  
\begin{prop}
 If $\ca{M}_H(k)$ is semi-simple and the hard Lefschetz theorem holds for $H$, then the conjecture $hC$ holds universally. If there exists an embedding $k \rightarrow \mathbb{C}$ and $H$ is the corresponding Betti cohomology, then $hD(H)$ is equivalent to $hC(H)$, and it holds if and only if every motivated cycle is homological. Moreover, if $hC(H)$ holds for Betti cohomology, then the comparison functor $\ca{M}_H(k) \rightarrow \ca{M}_k$ of Theorem~\ref{thm:properties}~(iii) is an equivalence.
  \end{prop}
  
  \begin{proof}
  Let $X$ be an irreducible smooth projective variety of dimension $d$. Let $X_0$ be the image of $\ell^{d} \colon [X] \rightarrow [X \otimes L^{-d}]$. Since $\bar{H} [X]$ is concentrated in degrees $\geq 0$ and  $\bar{H}[X \otimes L^{-d}]$ is concentrated in degrees $\leq 0$ we find that $\bar{H} X_0$ is concentrated in degree $0$. Since $\ca{M}_H$ is semi-simple, the epimorphism $[X]\rightarrow X_0$ is split, so $X_0$ is a direct summand of $[X]$. The hard Lefschetz theorem tells us that $\bar{H}[X] \rightarrow \bar{H}X_0$ is an isomorphism in degree zero. Taking duals and shifting, we get a direct sum decomposition $[X] \cong X_0 \oplus X^{\prime} \oplus X_{2d} $ such that $\bar{H} X_i$ is concentrated in degree $i$ and $\bar{H} X^{\prime}$ is concentrated in degrees $0<i<2d$. Applying the same reasoning to $\ell^{-(d-1)} \colon X^{\prime} \rightarrow X^{\prime} \otimes [L^{-(d-1)}]$ we get $X_1$ and $X_{2d-1}$ with a corresponding direct sum decomposition of $X^{\prime}$ and the conclusion follows by iterating.
  
  To see the second claim, note that $hC(H)$ implies $hB(H)$, which, as in the proof of Theorem~\ref{thm:CB}, implies that motivated cycles are homological. Thus the comparison functor $\ca{M}_H(k) \rightarrow \ca{M}_k$ of Theorem~\ref{thm:properties}~(iii) is an equivalence. Since $\ca{M}_k$ is semi-simple (as a consequence of the Hodge Index theorem, see \cite[Proposition~3.3]{ANDRE}), the category $\ca{M}_H(k)$ is semi-simple.
  \end{proof}
  
  We conclude with a brief remark concerning specialization from characteristic zero to characteristic $p > 0$. Let $R$ be a henselian DVR with field of fractions $K$ and residue field $k$, $\ell$ a prime different from $p=\mathrm{char}(k)$, and let $H$ respectively $H^{\prime}$ denote {\'e}tale cohomology with values in $\mathbb{Q}_\ell$-vector spaces for $\mathrm{SmProj}_K$ respectively $\mathrm{SmProj}_k$. Let $\ca{V}$ be the full subcategory of $\mathrm{SmProj}_{K}$ consisting of varieties with good reduction. Then there exists a functor $\mathrm{sp} \colon \mathrm{Mot}_H(\ca{V}) \rightarrow \mathrm{Mot}_{H^{\prime}}(k)$ (where the domain stands for the full subcategory of motives with good reduction) such that the triangle
  \[
  \xymatrix{ \mathrm{Mot}_H(\ca{V}) \ar[rr]^{\mathrm{sp}} \ar[rd]_{H} && \mathrm{Mot}_H(k) \ar[ld]^{H} \\ &(\mathbb{Z},\varepsilon)\mbox{-}\Vect_{\mathbb{Q}_{\ell}}}
  \]
  commutes up to natural isomorphism, see \cite[\S 3]{ANDRE_KAHN}: specialization of cycles is defined in \cite[\S 20.3]{FULTON} and the isomorphism follows from smooth proper base change for {\'e}tale cohomology; compatibility of this isomorphism with Chern classes follows from Grothendieck's \cite[Exp.~X~App.~7.]{SGA6}. An analogous triangle exists for de Rham cohomology and crystalline cohomology (here the compatibility is due to Messing, see \cite[Theorem~B.3.1]{GILLET_MESSING}). From the functoriality of the basic Tannakian factorization, we thus obtain a faithful exact symmetric strong monoidal functor $\bar{\mathrm{sp}} \colon \ca{M}_H(\ca{V}) \rightarrow \ca{M}_{H^{\prime}}(k)$ compatible with the fiber functors.
  
 If we restrict attention to the full subcategory $\ca{M}^{\prime}$ of $\ca{M}_H(\ca{V})$ generated by those $[X]$ for which $hC(X,H)$ holds in the category $\ca{M}_H(\ca{V})$, we get a  tensor functor $\bar{\mathrm{sp}} \colon \ca{M}^{\prime,\mathrm{tw}} \rightarrow \ca{M}_{\text{{\'e}t}, \ell}(k)$ of Tannakian categories.
 
 Andr{\'e} showed that every Hodge cycle on an abelian variety is motivated on a subcategory $\ca{V}^{\prime} \subseteq \mathrm{SmProj}_K$ of smooth projective varieties if this category contains certain abelian fibrations (building on Deligne's proof that all Hodge cycles are absolute Hodge cycles in this case), see \cite[\S 6]{ANDRE}. This now raises the question: can one choose the relevant abelian fibrations (which live over curves) so that they have good reduction, and so that the total space $X$ satisfies $hC(X,H)$ in $\ca{M}_H(\ca{V})$? If this question has an an affirmative answer, then we could conclude from Andr{\'e}'s theorem that every Hodge cycle between abelian varieties of CM-type comes from a morphism in $\ca{M}^{\prime,\mathrm{tw}}$, so we would get a tensor functor from Hodge structures of CM-type to $\ca{M}_{\text{{\'e}t}, \ell}(\bar{\mathbb{F}}_p)$ following \cite[\S 4]{MILNE} without assuming any further conjectures. To circumvent the problem of good reduction to some extent, it might be fruitful to study these questions in the context of mixed motives.